\documentclass{siamltex}
\usepackage{amsmath,amsfonts,amssymb,mathtools}
\usepackage{url}
\usepackage{latexsym}
\usepackage{graphicx,overpic}
\usepackage{subfigure,caption}
\usepackage{pgfplots}
\pgfplotsset{compat=newest}
\usepackage{color}
\usepackage{booktabs}
\usepackage{caption}
\usepackage{subcaption}
\usepackage{lscape}
\usepackage{array}
\newcolumntype{C}[1]{>{\centering\let\newline\\\arraybackslash\hspace{0pt}}m{#1}}
\usepackage{tikz}

\newcommand{\R}{\mathbb R}

\newcommand{\bL}{\mathbf L}

\newcommand{\bbH}{\Bbb{H}}
\newcommand{\bbY}{\Bbb{Y}}
\newcommand{\bbC}{\Bbb{C}}
\newcommand{\bbS}{\Bbb{S}}

\newcommand{\F}{\mathcal F}

\newcommand{\cP}{\mathcal P}

\newcommand{\rmax}{r_N^{\rm max}}

\newtheorem{remark}{Remark}[section]

\numberwithin{equation}{section}
\begin{document}

\title{On a maximum principle in discretization with spherical harmonics}
\author{
Arnold Reusken\thanks{Institut f\"ur Geometrie und Praktische  Mathematik, RWTH-Aachen
University, D-52056 Aachen, Germany (reusken@igpm.rwth-aachen.de).}
}
\maketitle
\begin{abstract}
Spherical harmonics have
been studied extensively and are used in many fields of application. We consider the probably earliest application of spherical harmonics, namely  for a representation of the solution of the Laplace equation on the $d$-dimensional unit sphere where the boundary data are represented in a spherical harmonics basis. We consider general $d \geq 2$. For the solution $u$ of the Laplace equation the fundamental maximum principle holds. If one approximates the boundary data in the finite dimensional space spanned by all spherical harmonics of degree at most $N$, the resulting discrete solution, denoted by $u_N$, does in general not satisfy the maximum principle. In this paper we derive sharp estimates for $\rmax$ such that the maximum principle holds for $u_N(x)$ for all $x \in \R^d$ with $\|x\| \leq \rmax$. We show that for $d=3$ one has $\rmax  \gtrsim 1-  3 N^{-1} \ln\big( 2^\frac13 N \big) $. Similar estimates are derived for $d \neq 3$. We also study the sharpness of these estimates. 
\end{abstract}


\section{Introduction}
Starting with the works of Laplace and Legendre about 250 years ago, spherical harmonics have
been studied extensively and are used in many fields of application. Spherical harmonics have become a standard simulation  tool in very different branches of 
the natural sciences and engineering, including geoscience, computer science, astrophysics, acoustics, quantum chemistry, electromagnetics and medical imaging 
\cite{ferrari2022spherical,vsprlak2016spherical,liu2025efficient,garboczi20173d,capilla2021analysis,wei2018generation,ledesma2020spherical,carlson2018utilizing,toth2011obtaining,cobos2023acoustic,bigi2023fast,tornese2024maximizing,velimsky2025electromagnetic,ha2022spharm}.
Furthermore, the concepts introduced by Laplace and Legendre have been extended in several directions and mathematical analyses have been developed; for an overview of the mathematical theory of spherical harmonics we refer  to \cite{DaiXu,Atkinson}.\\
In this paper we consider the probably earliest application of spherical harmonics, namely in the groundbreaking work of Laplace in which he derived  the Newton potential and showed that this potential is harmonic (outside sources). He used spherical harmonics to study  properties of harmonics on a spherical domain. Reformulated in more modern mathematical language, he studied (among many other things) the Laplace equation
\begin{equation} \label{eqLaplaceA}
   \Delta u= 0\quad   \text{in}~B^3, \quad u=f  \quad \text{on}~S^2,
\end{equation}
where $B^3$ is the (open) unit ball in $\R^3$ and $S^2=\partial B^3$ its boundary. As we know from the work of Laplace and Legendre, the spherical harmonics are a convenient tool for representing the harmonic function $u$ in terms of the boundary data $f$. We introduce some notation to explain this in more detail. Let $(Y_{n,j})_{1 \leq j\leq 2n+1}$ be the (normalized) spherical harmonics of degree $n$ (precise definitions are given in Section~\ref{secBasics}). Any $f \in L^2(S^2)$ can be represented in the basis of spherical harmonics, i.e., $f= \sum_{n=0}^\infty \sum_{j=1}^{2n+1} \alpha_{n,j}Y_{n,j}$ and the harmonic $u$ that solves \eqref{eqLaplaceA} is given by $u(r\xi)=\sum_{n=0}^\infty \sum_{j=1}^{2n+1}r^n \alpha_{n,j}Y_{n,j}(\xi)$, $\xi \in S^2$, $ r \in [0,1)$. In a numerical method based on this representation, which essentially is a spectral method, one replaces the series by a finite sum: $u_N(r\xi):=\sum_{n=0}^N \sum_{j=1}^{2n+1} r^n\alpha_{n,j}Y_{n,j}(\xi)$. For applications with (close to) spherical domains this basic idea can result in a feasible numerical method, cf. the references above. In a recent series of papers this technique is also applied to a Laplace equation on a domain consisting of a union of balls. Such a domain occurs in the COnductor-like Screening MOdel (COSMO)~\cite{Klamt}, which is a particular type of  continuum solvation model that plays a fundamental role  in computational chemistry.  We refer to the review articles~\cite{lipparini2016,tomasi2005quantum} for a thorough introduction to continuum solvation models. 
The COSMO approach uses the so-called van der Waals cavity, which models the solute's cavity as a union of balls, each of them centered around an atom. The solution of  a Laplace equation on such a cavity can be approximated using a Schwarz domain decomposition method. This  so-called ddCOSMO solver has been introduced in ~\cite{CaMaSt13} and is further extended in several directions, cf. \cite{Stammetal,Quan2018,Stamm2023} for an overview.  This approach, in which a spectral approximation with spherical harmonics on each spherical subdomain of the cavity plays a crucial role, leads to a very efficient numerical method for such a continuum solvation model.
 \\
A fundamental property of the Laplace equation is the maximum principle. Assume for simplicity that $f$ is continuous. Then for the solution $u$ of \eqref{eqLaplaceA} we have 
\begin{equation} \label{Max1}
   \min_{\xi \in S^2} f(\xi) \leq u(x) \leq \max_{\xi \in S^2} f(\xi)\quad \text{for all}~x \in B^3.
\end{equation}
It is clear that in general this property does not hold if we replace $u$ by $u_N$. \emph{The main topic of this paper is the analysis of a discrete maximum principle for $u_N$}. First note that due to $u_N(0)=u(0)$ the maximum principle trivially holds for any $N$ with $B^3$ in \eqref{Max1} replaced by $0$. In the approximation of $f$ by spherical harmonics the relatively large errors in the maximumnorm are dominated by high frequencies (cf. the Gibbs phenomenom for a discontinuous $f$). These high frequencies are strongly damped by the inverse Laplacian, where the size of this damping is a monotonically increasing function of $1-r$. Hence, it is plausible that for $u_N$ a maximum principle as in \eqref{Max1} may hold with $B^3$ replaced by a smaller ball $B^3(r_N):=\{\, x\in B^3\,|\,\|x\| \leq r_N\,\}$ and $0<r_N <1$. Indeed, in an analysis of a Schwarz domain decomposition method for solving a Laplace equation on a domain formed by the union of two overlapping discs, such a discrete maximum principle is derived for the two-dimensional case \cite{ReuskenDD}. One then uses the Fourier modes as spherical harmonics (cf. Section~\ref{secBasics}). In that paper it is shown that a discrete maximum principle holds for the discrete spectral approximation $u_N$ on a disc with radius $r_N = 1- \frac{\ln (2 N)}{N}$ (cf. Remark~\ref{rem2dspecial}). \\
In this paper we derive a similar discrete maximum principle result not only for the three-dimensional case \eqref{eqLaplaceA} but for arbitrary dimension $d \geq 2$. We outline the main result for the case $d=3$. Let $\rmax$ be the maximum radius such that for $u_N$ the discrete maximum principle holds on $B^3(\rmax)$. We will show that $\rmax \geq 1- 3 \, \frac{\ln (2^{\frac13} N)}{N}$ for $N \geq 8$ holds (Theorem~\ref{mainThm1} and Remark~\ref{remspecialcase}) and that this bound is asymptotically sharp (for $N \to \infty$) apart from a factor $\tfrac32$ (Corollary~\ref{corol1}). Note that the estmate is not an asymptotic one, but holds also for small $N$ values. 
For general dimension $d \geq 2$ we derive an estimate that is essentially of the form $\rmax \geq 1- d\, \frac{\ln( \tfrac{2}{e} d N)}{N}$ for not too small $N$ values (Remark~\ref{remspecialcase}). 
As far as we know such a discrete maximum principle, that we consider to be a fundamental property of  this spectral discretization method, has not been treated in the literature. The analysis is based on a representation of the spectral approximation $u_N$ using a projected version $G_N=G_N(r,t)$ of the continuous Poisson kernel $G(r,t)=\frac{1-r^2}{(1+r^-2rt)^{d/2}}$, $r \in [0,1)$, $t \in [-1,1]$. Positivity of this discrete Poisson kernel $G_N(r,t)$ for $r \in [0,r_N]$, $t \in [-1,1]$, can be directly related to the maximum principle on a ball in $\R^d$ with radius $r_N$. The main results outlined above will follow from an analysis of positivity of $G_N$.\\[1ex]
The remainder of the paper is organized as follows. In Section~\ref{secgeneral} we generalize the problem \eqref{Max1} to $d \geq 2$ dimensions and represent its solution using the Poisson kernel. In Section~\ref{secBasics} we collect results on spherical harmonics, known from the literature, that we need in our analysis. In Section~\ref{SecDiscrete} we derive the discrete Poisson kernel $G_N$ that describes the relation between the data $f$ and the discrete solution $u_N$. In Section~\ref{SecKernel} we derive some properties of $G_N$. In particular we show that there is a unique largest radius $r_N=:\rmax <1$ such that the discrete maximum principle holds on the ball with this radius. In our analysis we use the Lambert W function. Definition and properties of this function, known from the literature, are collected in Section~\ref{SecLambert}. The main results of the paper, namely lower and upper bounds for $\rmax$ are presented in Section~\ref{secbounds}. Finally, in Section~\ref{SecExp} we present results of a few numerical experiments that yield accurate estimates for $\rmax$ for the cases $d=2$ and $d=3$.

\section{Laplace equation and Poisson kernel in $\R^d$} \label{secgeneral}
Let $B^d \subset \R^d$, $d \geq 2$, be the open unit ball in $\R^d$ with boundary $S^{d-1}$. The analysis that we present holds for all $d \geq 2$, but in examples we  will be particularly interested in $d=2$ and $d=3$, which are the cases that are most relevant in applications. We consider the Laplace equation in $ B^d$ with data $f$ on $S^{d-1}$. We restrict to functions $f$ that are essentially bounded on $S^{d-1}$ and continuous almost everywhere on $ S^{d-1}$:
\begin{equation} \label{assdata}
  f \in \F:= L^\infty(S^{d-1}) \cap C(U), \quad \text{with}~{\rm meas}_{d-1}(S^{d-1}\setminus U)=0.
\end{equation}
Then the function
\begin{equation} \label{solution}
   u(x)=\frac{1}{\omega_{d-1}} \int_{S^{d-1}} k(x,\eta) f(\eta) \, d\eta,\quad k(x,\eta):= \frac{1- \|x\|^2}{\|x- \eta\|^d},
\end{equation}
with $\|\cdot\|$ the Euclidean norm in $\R^d$ and $\omega_{d-1}:=|S^{d-1}|=\frac{2 \pi^{\frac{d}{2}}}{\Gamma(\frac{d}{2})}$,
is well-defined and solves the the Laplace problem 
\begin{equation} \label{eqLaplace}
  \Delta u = 0 \quad \text{in}~~B^d,\quad u=f \quad \text{a.e. on}~~S^{d-1},
\end{equation}
cf. \cite{HackbuschElliptic}. 
Note that $\omega_1=2 \pi$, $\omega_2=4 \pi$. Using $x=r \xi$, $r\in [0,1)$, $\|\xi\|=1$, and 
\[ 
  \|x- \eta\|^2= \|r \xi - \eta\|^2= r^2 +1- 2 r \eta \cdot \xi, \quad \eta \in S^{d-1}, 
\]
results in the representation
\begin{equation} \label{solution2} \begin{split}
  u(r \xi) & = \frac{1}{\omega_{d-1}} \int_{S^{d-1}} G(r, \eta \cdot \xi) f(\eta) \, d\eta, \quad r \in [0,1), ~\xi \in S^{d-1}, \\
  G(r,t) & :=  \frac{1-r^2}{(1+r^2 -2 r t)^\frac{d}{2}}  \quad  r \in [0,1), ~t \in [-1,1].
\end{split} \end{equation}
The maximum principle for the Laplace problem \eqref{eqLaplace} is equivalent to the following two properties of the Poisson kernel $G$:
\begin{align} 
  G(r,t) & \geq 0 \quad \text{for all}~(r,t) \in [0,1) \times [-1,1], \label{max1} \\
  \frac{1}{\omega_{d-1}} \int_{S^{d-1}} G(r, \eta \cdot \xi) \, d\eta &=1 \quad \text{for all}~ r \in [0,1),~\xi \in S^{d-1}. \label{max2} 
\end{align}
The property \eqref{max2} follows from the following argument. If in \eqref{eqLaplace} we take for $f$ the constant function $1$, the solution $u$ is the constant function $1$. Now use the representation formula \eqref{solution2}.
Below we consider an approximation of the boundary data by spherical harmonics, $f_N\approx f$, and study a maximum principle for the solution $u_N$ of the Laplace problem \eqref{eqLaplace} with $f$ replaced by $f_N$. 

\section{Basics on spherical harmonics} \label{secBasics}
We introduce notation and collect basic definitions and results known from the literature, e.g.~\cite{Atkinson}.
The spherical harmonics are defined as the (suitably scaled) eigenfunctions of the Laplace-Beltrami operator on $S^{d-1}$. Another possibility to introduce these functions is via homogeneous polynomials in $\R^d$. The latter approach is more constructive and allows the (easier) derivation of several important properties of spherical harmonics. Below we summarize definitions and known results, based on the exposition in \cite{Atkinson}, that we need in the analysis of the discrete Poisson kernel. 

The space of homogeneous polynomials in $\R^d$ of degree $n$ is given by 
\[
  \bbH_n^d=\{\, p(x)=\sum_{|\alpha|=n} a_\alpha x^\alpha\,|\, a_\alpha \in \bbC,~x \in \R^d\,\},
\]
with $\alpha=(\alpha_1,\ldots ,\alpha_d)$ the usual multi-index notation. For example, $\bbH_3^2=\{\, a_1x_1^3+a_2x_1^2x_2+a_3x_1 x_2^2 +a_4x_2^3\,|\, a_j \in \bbC\,\}$. Note that for $p \in \bbH_n^d$ there holds $p(\beta x)=\beta^n p(x)$ for all $\beta \in \bbC,~x \in \R^d$.   The \emph{homogeneous harmonics of degree $n$} are given by
\[
  \bbY_n(\R^d)=\{\, p \in \bbH_n^d\,|\, \Delta p=0\,\},
\]
and based on this we define the space of \emph{spherical harmonics of degree $n$}:
\[
  \bbY_n^d:=\bbY_n(\R^d)_{|S^{d-1}}.
\]
The  dimension of this space is denoted by $\nu_{n,d}$ and we have $\nu_{0,d}={\rm dim}(\bbY_0^d)=1$, 
\begin{equation} \label{defnu} \nu_{n,d}={\rm dim}(\bbY_n^d)=\frac{(2n+d-2)(n+d-3)!}{n! (d-2)!}, \quad n \geq 1.
\end{equation}
One may check $\nu_{n,2}= 2$ for $n \geq 1$, $\nu_{n,3}=2n+1$ for all $n \geq 0$. For $d \geq 4$ and $n \geq 1$ we note
\[ \begin{split}
\nu_{n,d} & = \frac{1}{(d-2)!} (2n+d -2)\prod_{k=1}^{d-3}(n+k)=\frac{1}{(d-2)!} n^{d-2} \big(2+\frac{d -2}{n}\big) \prod_{k=1}^{d-3}(1+\frac{k}{n})  \\ & \leq 
\frac{1}{(d-2)!} n^{d-2} d \prod_{k=1}^{d-3} (1+k) = d n^{d-2}.
\end{split} \]
Hence,
\begin{equation} \label{boundnu} \begin{split}
  \nu_{n,d} & \leq C_d (n+1)^{d-2} , \quad \text{for all}~ n \geq 1,~d \geq 2, \\ & \text{with}~ C_2=C_3=2,~C_d=d\quad \text{for}~d \geq 4. 
\end{split} \end{equation}
The generating function corresponding to $(\nu_{n,d})_{n \geq 0}$ is given by (\cite[(2.13)]{Atkinson})
\begin{equation} \label{genfunction} \sum_{n=0}^\infty \nu_{n,d} z^n = \frac{1+z}{(1-z)^{d-1}}\quad \text{for}~|z|<1. 
\end{equation}
Below, $\xi \in \R^d$ will always satisfy $\|\xi\|=1$, i.e., $\xi \in S^{d-1}$. Note that any spherical harmonic $Y_n \in \bbY_n^d$ is related to a unique homogeneous harmonic $H_n \in \bbH_n^d$ by $H_n(r \xi)=r^n Y_n(\xi)$.\\
 A relevant subset of $\bbY_n(\R^d)$ is the one of \emph {Legendre harmonics} $\bL_{n,d}: \R^d \to \R$,  consisting of \emph{real valued} $\bL_{n,d} \in  \bbY_n(\R^d)$ that satisfy $\bL_{n,d}(e_d)=1$ and $\bL_{n,d}(Q x)=\bL_{n,d}(x)$ for all  $x\in \R^d$ and all
orthogonal  matrices $Q \in \R^{d \times d}$ that satisfy $Qe_d=e_d$, with $e_d=(0,\ldots,0,1)^T \in \R^d$ (cf. \cite[Section 2.1.2]{Atkinson}).
Arbitrary $\xi \in S^{d-1}$ and $t \in [-1,1]$ are uniquely related by 
\[
   \xi = \sqrt{1-t^2} \begin{pmatrix} \hat \xi \\ 0 \end{pmatrix} + \begin{pmatrix} 0 \\ t\end{pmatrix}, \quad \hat \xi \in S^{d-2},\, 0 \in \R^{d-1},   
\]
and the \emph{Legendre polynomial of degree $n$} is defined by
\[
  L_{n,d}(t):= \bL_{n,d}(\xi), \quad t\in [-1,1].
\]
For $x =r \xi \in B^d$ we then have the relation
\[
  \bL_{n,d}(x)=\bL_{n,d}(r \xi)=r^n \bL_{n,d}(\xi)=r^n L_{n,d}(t).
\]
\begin{remark} \label{remex1} \rm
For the cases $d=2$, $d=3$, these Legendre polynomials are the Chebyshev and classical Legendre polynomials:
\begin{align} 
   L_{n,2}(t)& =\cos (n \arccos t)=:T_n(t).\label{Legendredistwo}\\
   L_{n,3}(t)& = \frac{1}{2^n n!} \left(\frac{d}{dt}\right)^n (t^2-1)^n =:L_n(t). \label{Legendredisthree}
\end{align}
\end{remark}

We recall a few fundamental results. Below orthogonality is with respect to the complex $L^2$ scalar product on $S^{d-1}$, i.e., $(f,g)_{S^{d-1}}=\int_{S^{d-1}} f(\xi)\overline{g(\xi)}\, d\xi$. By $L^2(S^{d-1})$ we denote the space of complex-valued  squared Lebesque
integrable functions on $S^{d-1}$.
\begin{lemma} \label{LemAddition} (Addition theorem \cite[Theorem 2.9]{Atkinson}) Let $(Y_{j})_{1 \leq j \leq \nu_{n,d}}$ be an orthonormal basis of $\bbY_n^d$. The following holds:
\begin{equation} \label{addition}
  \sum_{j=1}^{\nu_{n,d}} Y_{j}(\xi) \overline{Y_{j}(\eta)}=\frac{\nu_{n,d}}{\omega_{d-1}} L_{n,d}(\xi \cdot \eta), \quad \xi,\eta \in S^{d-1}.
\end{equation}
\end{lemma}
A direct consequence of this lemma is the following. Define, for given $\xi \in S^{d-1}$ and $n,d\in \Bbb{N}$, $\ell_{n,d,\xi}(\eta):=L_{n,d}(\xi \cdot \eta)$, $\eta \in S^{d-1}$. The result \eqref{addition} implies that $\ell_{n,d,\xi} \in \bbY_n^d$ holds.
Another result, often called the Poisson identity, establishes a relation between  the Poisson kernel  $G$ in \eqref{solution2} and the Legendre polynomials. 
\begin{lemma} \label{LemPoisson} (Poisson identity \cite[Proposition 2.28]{Atkinson})
 The following holds:
 \begin{equation} \label{Poissonid}
   G(r,t)= \sum_{n=0}^\infty \nu_{n,d} r^n L_{n,d}(t), \quad (r,t) \in [0,1) \times [-1,1]. 
 \end{equation}
\end{lemma}
\begin{remark} \label{remex2} \rm 
For the cases $d=2$, $d=3$, we thus have the following representations for the Poisson kernel, cf. Remark~\ref{remex1}:
\begin{align}
 d=2:\quad G(r,t) &= \frac{1-r^2}{1+r^2-2 rt} = 1+2\sum_{n=1}^\infty r^n T_n(t), \label{form1}\\
 d=3:\quad G(r,t) &= \frac{1-r^2}{(1+r^2-2 rt)^{\frac32}} = \sum_{n=0}^\infty (2n+1) r^n L_n(t). \label{form2}
\end{align}
\end{remark}

\begin{lemma}(Orthogonal decomposition \cite[Theorem 2.38]{Atkinson}) We have the {\rm orthogonal decomposition}
 \begin{equation} \label{decomposition}
  L^2(S^{d-1})= \bigoplus_{n=0}^\infty \bbY_n^d.
 \end{equation}
\end{lemma}
We introduce the orthogonal projection $P_n: L^2(S^{d-1}) \to \bbY_n^d$. If $(Y_{n,j})_{1 \leq j \leq \nu_{n,d}}$ is an orthonormal basis of $\bbY_n^d$ this projection has the representation
\begin{equation} \label{projection}
 (P_n f)(\xi)= \sum_{j=1}^{\nu_{n,d}} (f,Y_{n,j})_{S^{d-1}} Y_{n,j}(\xi), \quad \xi \in S^{d-1}. 
\end{equation}
Due to the orthogonality property we have $P_nP_m=\delta_{nm}$. 
Using the addition theorem (Lemma~\ref{LemAddition}) one obtains the following alternative  representation of this projection.
\begin{lemma} \label{simpleP}
 For the orthogonal projection $P_n$ we have the representation
 \begin{equation} \label{simplePform}
  (P_n f)(\xi)= \frac{\nu_{n,d}}{\omega_{d-1}} \int_{S^{d-1}} L_{n,d}(\xi \cdot \eta) f(\eta)\, d\eta, \quad \xi \in S^{d-1}.
 \end{equation}
\end{lemma}
\begin{proof}
 We use \eqref{addition} and obtain
 \begin{align*}
  (P_n f)(\xi)& =\sum_{j=1}^{\nu_{n,d}} (f,Y_{n,j})_{S^{d-1}} Y_{n,j}(\xi) \\
  & = \int_{S^{d-1}} \sum_{j=1}^{\nu_{n,d}} Y_{n,j}(\xi)\overline{Y_{n,j}(\eta)} \, f(\eta)\, d\eta \\
   & = \frac{\nu_{n,d}}{\omega_{d-1}}\int_{S^{d-1}} L_{n,d}(\xi \cdot \eta)\, f(\eta)\, d\eta.
 \end{align*}
\end{proof}
\ \\
The result \eqref{simplePform}  establishes a fundamental relation between Legendre polynomials and the space of spherical harmonics. Note that in the representation \eqref{simplePform} we do not not need a basis of $\bbY_n$ and that the Legendre polynomial $L_{n,d}$ depends on the \emph{scalar} argument $\xi \cdot \eta$. 

Finally we introduce the \emph{space of spherical harmonics of degree at most $N$}:
\[
  \bbS_N^d:= \bigoplus_{n=0}^N \bbY_n^d. 
\]
The orthogonal projection on this space is denoted by $\cP_N:\, L^2(S^{d-1}) \to \bbS_N^d$. For this projection we have $\cP_N = \sum_{n=0}^N P_n$. 
\section{Spherical harmonics discretization and the discrete Poisson kernel} \label{SecDiscrete}
For the discretization of the Laplace problem we note that for the data $f$ we have $f\in L^2(S^{d-1})$, cf. \eqref{assdata}. In the discretization we 
replace $f$ by $\cP_Nf$.  Hence, the discrete solution $u_N \in C^\infty(B^d)$ that approximates $u$ is given by
\begin{equation} \label{eqLaplaceDiscrete}
 \Delta u_N = 0 \quad \text{in}~~B^d,\quad u_N=\cP_Nf \quad \text{on}~~S^{d-1}.
\end{equation}
In an implementation of the discretization \eqref{eqLaplaceDiscrete} one usually determines (approximately) $P_Nf$ in an orthonormal spherical harmonics basis $(Y_{n,j})_{1 \leq j \leq \nu_{n,d}}$, $0 \leq n \leq N$. There is a general technique for constructing a suitable orthonormal basis in any dimension $d \geq 2$, cf. \cite[Section 2.11]{Atkinson}. We then have the representation:
\begin{equation} \label{algostep1}
 \cP_Nf = \sum_{n=0}^N \sum_{j=1}^{\nu_{n,d}} \alpha_{n,j} Y_{n,j}, \quad \alpha_{n,j} =(f, Y_{n,j})_{S^{d-1}}.
\end{equation}
The corresponding harmonic on $B^d$ is  given by
\begin{equation} \label{algostep2}
 u_N(r \xi) = \sum_{n=0}^N r^n \sum_{j=1}^{\nu_{n,d}} \alpha_{n,j} Y_{n,j}, \quad r \in [0,1),~\xi \in S^{d-1}.
\end{equation}
The computational work needed for a discretization based on \eqref{algostep1}-\eqref{algostep2} is dominated by  the quadrature for approximation of the  integrals $(f, Y_{n,j})_{S^{d-1}}$. For $d=2$ we write $\xi=(\cos \theta, \sin \theta)$ and an  orthogonal basis of $Y_n^2$, $n \geq 1$, is given by $Y_{n,1}(\xi)= \cos (n \theta)$, $Y_{n,2}(\xi)= \sin (n \theta)$. Using suitable quadrature points the integrals $(f, Y_{n,j})_{S^{1}}$ can be approximated very efficiently using the fast Fourier transform. For $d=3$ the classical spherical harmonics form an orthonormal system and the Lebedev quadrature, which is designed specifically for accurate integration of spherical harmonics, is often used.

Clearly,  for the discrete solution $u_N$ a maximum principle (for simplicity we assume $f \in C(S^{d-1})$) of the form  $\min_{\xi \in S^{d-1}} f(\xi) \leq u_N(x) \leq \max_{\xi \in S^{d-1}} f(\xi)$ for all $x \in B^d$ does in general not hold. Let $B^d(r)$ denote the ball with center 0 and radius $r$ in $\R^d$. In the remainder of the paper we show that such a maximum principle does hold if one replaces $x \in B^d$ by $x \in B^d(r_N^\ast)$,  and we derive lower and upper estimates that reveal  how the maximal radius $r_N^\ast$ depends on the spherical harmonics degree $N$. 

The analysis is based on a study of the \emph{discrete  Poisson kernel}, that we now introduce. Note that the discrete solution $u_N$ can be represented as 
\begin{equation} \label{LaplaceSol}
 u_N(x) = \frac{1}{\omega_{d-1}} \int_{S^{d-1}}G(r, \xi \cdot \eta) \cP_Nf(\eta)\, d\eta, \quad x=r \xi \in B^d,
\end{equation}
cf. \eqref{eqLaplaceDiscrete} and \eqref{solution2}. 
We can shift the orthogonal projection to the kernel function $G$. A formula for the  discrete Poisson kernel, denoted by $G_N$,  is derived in the following lemma.
\begin{lemma} 
Let $f \in \F$. The following identity holds for all $\xi \in S^{d-1}$, $r \in [0,1)$:
\begin{equation} \label{idGN} \begin{split}
 \int_{S^{d-1}}G(r, \xi \cdot \eta) \cP_N f(\eta)\, d\eta & = \int_{S^{d-1}}G_N(r, \xi \cdot \eta) f(\eta)\, d\eta, \\\ \text{with}~~ G_N(r,t)& := \sum_{n=0}^N \nu_{n,d} r^n L_{n,d}(t) \qquad \text{{\rm(}$\nu_{n,d}$ defined in\eqref{defnu}{\rm)}}.
\end{split} \end{equation}
\end{lemma}
\begin{proof}
 Let $f \in \F$, $\xi \in S^{d-1}$ and $r \in [0,1)$ be given. Recall $\cP_N=\sum_{m=0}^N P_m$ and $G(r, \xi \cdot \eta)=\sum_{n=0}^\infty \nu_{n,d} r^n L_{n,d}(\xi \cdot \eta)$. For $\ell_{n,d,\xi}(\eta):=L_{n,d}(\xi \cdot \eta)$ we have $\ell_{n,d,\xi} \in \bbY_n^d$. Due to the orthogonality property \eqref{decomposition} we have $P_n \ell_{n,d,\xi}=\ell_{n,d,\xi}$ and $P_m \ell_{n,d,\xi}=0$ for $n \neq m$. Using this we obtain
 \begin{align*}
  \int_{S^{d-1}}G(r, \xi \cdot \eta) \cP_Nf(\eta)\, d\eta &= \sum_{n=0}^\infty \nu_{n,d} r^n\int_{S^{d-1}} \ell_{n,d,\xi}(\eta)\cP_Nf(\eta)\, d\eta \\
  &=   \sum_{n=0}^\infty \nu_{n,d} r^n \sum_{m=0}^N \int_{S^{d-1}}P_m\ell_{n,d,\xi}(\eta)f(\eta)\, d\eta \\
  &= \sum_{n=0}^N \nu_{n,d} r^n \int_{S^{d-1}}\ell_{n,d,\xi}(\eta)f(\eta)\, d\eta\\
  &=\int_{S^{d-1}}  \sum_{n=0}^N \nu_{n,d} r^n  L_{n,d}(\xi \cdot \eta) f(\eta)\, d\eta ,
 \end{align*}
which yields the result \eqref{idGN}. 
\end{proof}

The proof above shows that $G_N$ is the \emph{projected kernel}, $G_N=\cP_NG$, in the following sense. For given $r \in [0,1)$, $\xi \in S^{d-1}$, define $G_{r,\xi}: \, \eta \to G(r,\xi \cdot \eta)$, $\eta \in S^{d-1
}$. Then $G_N(r,\xi\cdot \eta)=(\cP_N G_{r,\xi})(\eta)$ holds. 
The projected kernel function $G_N$ determines the dependence of $u_N$ on the \emph{original} problem data $f$:\begin{equation} \label{repres1discr}
  u_N(x) = \frac{1}{\omega_{d-1}} \int_{S^{d-1}}G_N(r, \xi \cdot \eta) f(\eta)\, d\eta, \quad x=r \xi \in B^d. 
\end{equation}.
\section{Properties of the discrete Poisson kernel $G_N$} \label{SecKernel}
In this section we study properties of $G_N$. 
First note that for the kernel $G$ we have the fundamental properties \eqref{max1}-\eqref{max2} that are equivalent to the maximum priniciple of the Laplace problem. 
 We now discuss these two properties for  the discrete Poisson kernel $G_N$. The scaling property \eqref{max2} also holds for $G_N$, since using $G_N=\cP_N G$ we obtain  for arbitrary $N \in \Bbb{N}$:
\begin{equation} \label{maxcond2GN} \begin{split}
 & \frac{1}{\omega_{d-1}}\int_{S^{d-1}} G_N(r,\xi\cdot\eta)\, d\eta  =\frac{1}{\omega_{d-1}}\int_{S^{d-1}} G(r,\xi\cdot\eta) \cP_N1 \, d\eta \\ & = \frac{1}{\omega_{d-1}}\int_{S^{d-1}} G(r,\xi\cdot\eta)d\eta=1 \quad\text{for $\xi \in S^{d-1}$}, ~0 \leq r <1. 
\end{split}
\end{equation}
 The positivity property $\min_{|t|\leq 1} G_N(r,t) \geq 0$ for all $r \in [0,1)$,  however, does not hold. To illustrate this we show the discrete kernel $G_N(r,t)$ for $N=7$ and $N=20$ in Figure~\ref{FigKN}. 
\begin{figure}[ht!]
	\hspace*{-0.3cm}\centering\includegraphics[width=0.53\textwidth]{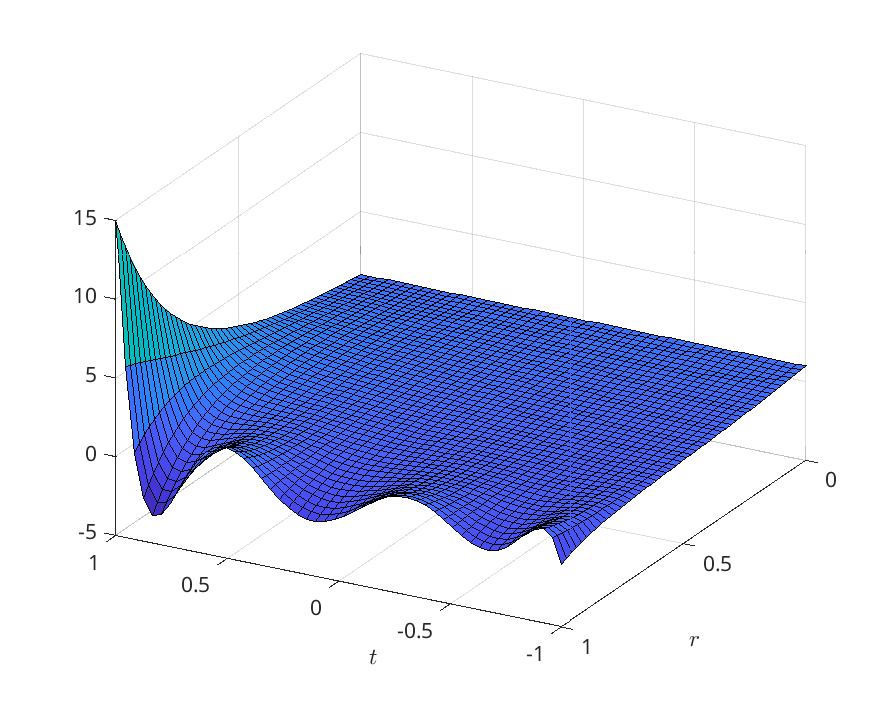} \hspace*{-0.8cm}
	\includegraphics[width=0.53\textwidth]{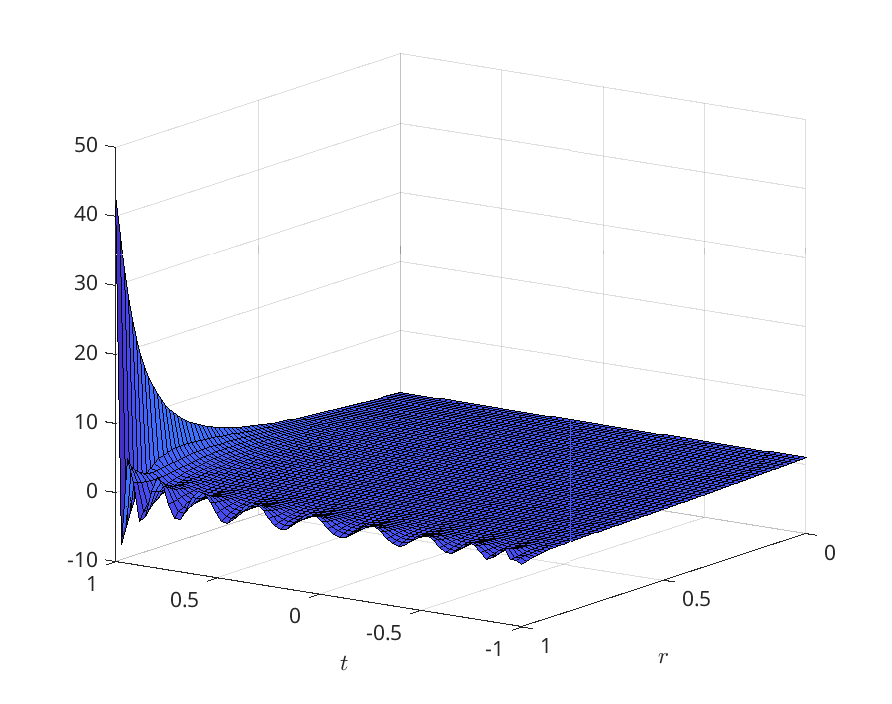}
	\caption{Discrete Poisson kernels $G_N(r,t)$ for $N=7$ (left) and $N=21$ (right). \label{FigKN}}
\end{figure}
\ \\
 The main topic of this paper is to answer the following question:\\
 \emph{What is the largest radius $\rmax$ for which
\begin{equation} \label{claim}
 \min_{|t|\leq 1} G_N(r,t) \geq 0 \quad \text{for}~0 \leq r \leq \rmax
\end{equation}
holds, and how does this radius depend on $N$?}\\
In the remainder of this section we show  that a unique $ \rmax$ exists and derive first (rather inaccurate) lower and upper bounds for $1- \rmax$. In Section~\ref{secbounds} we derive sharp lower and upper bounds for $1- \rmax$ that in particular reveal how $\rmax$ tends to 1 depending on $N$. \\
 We first derive a few basic propertes related to \eqref{claim}, where we use the following (cf. \cite[Sect. 2.7.5]{Atkinson}:
 \begin{equation} \label{basicprop} \begin{split}
 & L_{n,d}(1)=1,\quad L_{n,d}(-t)=(-1)^nL_{n,d}(t) \quad \text{for all}~t \in[-1,1],~ d\geq 2, ~n \geq 0,\\
 & \max_{t \in [-1,1]}|L_{n,d}(t)| = 1, \quad \text{for all}~d\geq 2,~n\geq 0.
 \end{split} \end{equation}

\begin{lemma} \label{lemClaim}
For $r_0 \in (0,1)$  we have:
\begin{equation} \label{ResClaim} \min_{|t|\leq 1} G_N(r_0,t) \geq 0 \quad \Rightarrow \quad \min_{|t|\leq 1} G_N(r,t) \geq 0 \quad \text{for all}~r \in [0,r_0].
\end{equation}
The following holds:
\begin{equation} \label{lower}
 \min_{|t|\leq 1} G_N(r,t) \geq 0 \quad \text{for all}~ r \in [0,\tfrac{1}{2d-1}].
 \end{equation}
 For any odd $N \geq 1$ and with $r_N^\ast:=1-\frac {1}{2N+\frac12 d}$ we have 
\begin{equation} \label{Res5}
 \min_{|t|\leq 1} G_N(r,t) < 0 \quad \text{for all}~ r \in [r_N^\ast,1].
\end{equation}
\end{lemma}
\begin{proof}  Note that $(r,t) \to G_N(r,t)$ is continuous on $[0,1]\times [-1,1]$. 
For given $N$ let   $r_0 \in (0,1)$ be such that $ \min_{|t|\leq 1} G_N(r_0,t) \geq 0 $ holds. Consider the discrete Laplace problem \eqref{eqLaplaceDiscrete} with data $f \in C(S^{d-1})$, $f \geq 0$. Its solution $u_N$, given by \eqref{repres1discr}, satisfies
\[  \Delta u_N=0 \quad \text{in}~B^d, \quad  u_N= \cP_Nf \quad \text{on}~S^{d-1}.
\]
From \eqref{repres1discr}, $f \geq 0$ and $G_N(r_0,t) \geq 0$ for all $t \in [-1,1]$ it follows that $0 \leq u_N(x) \leq \max_{\xi \in S^{d-1}}f(\xi)$ holds for all $x$ with $\|x\|=r_0$. Since $u_N$ is harmonic, the maximum principle yields
\begin{equation} \label{MaxP}
 0 \leq u_N(x) \leq \max_{\xi \in S^{d-1}}f(\xi) \quad \text{for all}~ \|x\|\leq r_0.
\end{equation} 
 Assume that  there exists $\hat r \in [0,r_0)$ with $\min_{|t|\leq 1} G_N(\hat r,t) < 0$. Hence $G_N(\hat r,\hat t) < 0$ for some $\hat t \in [-1,1]$. Choose $\xi \in S^{d-1}$ and define $x:= \hat r \xi$. 
 For $\delta >0$ sufficiently small, define  $S_\delta:= \{\, \eta \in S^{d-1}\,|\, 
\xi \cdot \eta \in [\hat t -\delta, \hat t+\delta]$\,\}. Choose $f  \in C(S^{d-1})$, $f \geq 0$, with $f > 0$ on $S_{\frac12\delta}$ and $f=0$ on $S^{d-1} \setminus S_\delta$. For $\delta >0$ sufficiently small we then have
 $u_N(x)= \frac{1}{\omega_{d-1}}\int_{S^{d-1}} G_N(\hat r,\xi \cdot \eta) f(\eta)\, d \eta <0$, which contradicts the left inequality in \eqref{MaxP}.  Hence,  \eqref{ResClaim} holds. \\
From $\|L_{n,d}\|_\infty=1$ and  \eqref{genfunction} it follows that, for arbitrary $t \in [-1,1]$ we have
\begin{align*}
  G_N(r,t) & = 1+ \sum_{n=1}^N \nu_{n,d} r^n L_{n,d}(t) \geq 1-  \sum_{n=1}^\infty \nu_{n,d} r^n =2-  \sum_{n=0}^\infty \nu_{n,d} r^n\\ &  = 2-\frac{1+r}{(1-r)^{d-1}}= \frac{2(1-r)^{d-1}-r-1}{(1-r)^{d-1}}.
\end{align*}
Now note
\[
  2(1-r)^{d-1}-r-1 \geq 2\big(1-(d-1)r\big) -r-1=1-(2d-1)r \geq 0 \quad \text{for}~ r \in [0,\tfrac{1}{2d-1}],
  \]
which proves the result \eqref{lower}.
For proving \eqref{Res5}  we will show that for any odd $N \geq 1$ and with $r_N^\ast=1-\frac{1}{2N+\frac12 d}$
\begin{equation} \label{hulpRes5}
 \min_{|t|\leq 1} G_N(r_N^\ast,t) < 0
\end{equation}
holds.  The result \eqref{Res5} then follows from \eqref{hulpRes5} and \eqref{ResClaim}.
For deriving \eqref{hulpRes5} we note 
\[
 \min_{|t|\leq 1} G_N(r,t) \leq G_N(r,-1)= \sum_{n=0}^{N} \nu_{n,d} r^n L_{n,d}(-1)= \sum_{n=0}^{N} \nu_{n,d} (-r)^n
\]
for all $r \in [0,1)$. 
We first treat the case $d=2$. For this we obtain, using $\nu_{0,2}=1$, $\nu_{n,2}=2$ for $n \geq 1$, and the assumption that $N$ is odd:
\[
  \sum_{n=0}^{N} \nu_{n,d} (-r)^n = -1 + 2 \sum_{n=0}^N (-r)^n = \frac{1-r - 2 r^{N+1}}{1+r}.  
\]
For the nominator with $r=r_N^\ast= 1- \frac{1}{2N+1}$ we obtain
\[
  \frac{1}{2N+1} - 2\big( 1 -\frac{1}{2N+1}\big)^{N+1} \leq  \frac{1}{2N+1} - 2\big( 1 -\frac{1}{2N+1}\big)^{2N+1}=:f(2N+1),
\]
with $f(x)=\frac{1}{x} - 2 \big(1 - \frac{1}{x}\big)^x$. From $f'(x) \leq 0$ for $x \geq 3$ it follows that $f(2N+1) \leq f(3) <0$, and thus \eqref{hulpRes5} holds for $d=2$. For $d \geq 3$ we note
\begin{equation} \label{hh} \begin{split}
  \sum_{n=0}^{N} \nu_{n,d} (-r)^n & = \sum_{n=0}^{\frac12 (N-1)}  \nu_{2n,d} r^{2n} - \sum_{n=0}^{\frac12 (N-1)}  \nu_{2n+1,d} r^{2n+1} \\ & = \sum_{n=0}^{\frac12 (N-1)}  \nu_{2n,d} r^{2n}\big( 1- r \frac{\nu_{2n+1,d}}{\nu_{2n,d}}\big).
\end{split} \end{equation}
For the term with $n=0$ in the sum we have $ 1- r_N^\ast \frac{\nu_{1,d}}{\nu_{0,d}}= 1- r_N^\ast d \leq 1-\tfrac12 d <0$, and for $ 1 \leq n \leq \tfrac12( N-1)$ (using $d \geq 3$):
\[
  \frac{\nu_{2n,d}}{\nu_{2n+1,d}}= \frac{(2n+1)(4n+d-2)}{(2n+d-2)(4n+d)} \leq \frac{4n+d-2}{4n+d} < \frac{4N+d-2}{4N+d} = r_N^\ast.
\]
Hence we obtain $ 1- r_N^\ast \frac{\nu_{2n+1,d}}{\nu_{2n,d}} <0$ for all $n=0,1, \ldots, \tfrac12( N-1)$. Using this in \eqref{hh} completes the proof of \eqref{hulpRes5}  for $d \geq 3$.
\end{proof}
\ \\[1ex]
The results of this theorem show that there is a unique \emph{maximal} radius $r$, denoted by $\rmax$, such that \eqref{claim} holds. For any $r<1,\, t \in [-1,1]$ we have $\lim_{N \to \infty} G_N(r,t)=G(r,t)$ and thus $\lim_{N\to \infty }\rmax=1$ holds.
From \eqref{lower} and \eqref{Res5} it follows that
\begin{equation} \label{apriori}
 1- \rmax  \leq 1- \frac{1}{2d-1} \quad \text{for all $N$}, \quad \frac{1}{2N +\frac12 d} \leq 1-\rmax \quad \text{for $N$ odd}. 
\end{equation}
In Section~\ref{secbounds} below  we derive (sharp) upper and lower bounds for $1-\rmax$, which  
in particular reveal how fast this difference tends to zero as a function of the discretization parameter $N$.

\section{Basics on Lambert W function} \label{SecLambert}
We recall a few basic properties of the Lambert W function that we will use in the analysis.
The function $h:\, [-1, \infty) \to [-e^{-1}, \infty)$, $h(z):=z e^z$, is a monotone bijection. Its inverse function
is the Lambert $W$ function, cf. \cite{Corless1996}, $W:\, [-e^{-1}, \infty) \to [-1, \infty)$, with $W(h(z))=z$ for all $z \geq [-1, \infty)$ and $ W(z)e^{W(z)}=z$ for all $z \geq [-e^{-1}, \infty)$. The following (sharp) estimates for the Lambert function are from \cite[Theorem 2.1]{Hoorfar2008}:
\begin{equation} \label{EstiW}
   \ln z - \ln \ln z \leq W(z) \leq \ln z -\tfrac12  \ln \ln z \quad \text{for}~~z \geq e.
\end{equation}
From this result it follows that $\lim_{z \to \infty}\frac{W(z)}{\ln z}=1$ holds. 

Below we are interested in determining a possibly large $r^\ast(N)$ such that  
\begin{equation} \label{basic3}
  r^{ N} \leq 1- r \quad \text{for all}~r \in [0,r^\ast(N)].
\end{equation}
The best possible $r^\ast(N)$ is $r_0=r_0(N)$ with $r_0$ such that $r_0^{N}= 1- r_0$. 
We study the asymptotic dependence of $r_0(N)$ on $N$.
\begin{lemma} \label{lemmaW1}
Let $r_0=r_0(N)$ be such that $r_0^N=1-r_0$, $N\in \Bbb{N}$. The following holds:
\begin{equation} \label{estW1}
 \lim_{N\to \infty} \frac{1-r_0(N)}{\frac{1}{N} \ln N} =1.
\end{equation}
Furthermore,   \eqref{basic3} is satisfied for
\begin{equation} \label{estW2}
  r^\ast(N):= 1- \frac{W(N)}{N}.
\end{equation}
\end{lemma}
\begin{proof} We first show that $r^\ast=r^\ast(N)$ as in \eqref{estW2} is a lower bound for $r_0$, where $r_0=r_0(N)$ is the unique solution in the interval $(0,1)$ of $r^N=1-r$.
 Hence, $r_0$ is the unique positive solution of $e^{N \ln r}= 1- r$. Using $\ln r \leq r-1$ (for $r \in (0,1]$) we obtain $e^{N \ln r} \leq e^{-N(1-r)}=:g_N(r)$ and due to monotonicity the unique solution $\hat r_0 \in (0,1]$ of $g_N(r)=1-r$ satisfies $\hat r_0 \leq r_0 \leq 1$. Now note
\[ \begin{split}
 g_N(\hat r_0) & =1-\hat r_0 ~\Leftrightarrow ~ e^{-N(1-\hat r_0)} = 1-\hat r_0  ~\Leftrightarrow ~N (1-\hat r_0)e^{N( 1-\hat r_0)}= N \\ & \Leftrightarrow ~h\big(N (1-\hat r_0)\big)=N\Leftrightarrow ~1-\hat r_0 = \frac{W(N)}{N}.
\end{split}\]
We conclude $\hat r_0=r^\ast$ and thus $r^\ast \leq r_0$ holds. Combining this with a monotonicity argument proves the result \eqref{estW2}. Since $\lim_{N \to \infty} \frac{W(N)}{N}=1$, cf. \eqref{EstiW}, this implies  $\lim_{N \to \infty} r_0(N)=1$.
We introduce the notation $q_N:=\frac{1-r_0(N)}{\frac{1}{N}\ln N}$. Using  $r^\ast \leq r_0$ and $\lim_{N \to \infty} \frac{W(N)}{N}=1$ we obtain
\begin{equation} \label{P1}
 q_N \leq \frac{1-r^\ast(N)}{\frac{1}{N}\ln N} \to 1 \quad \text{for}~N \to \infty.
\end{equation}
This in particular implies that the sequence $(q_N)_{N \geq 1}$ is bounded, i.e., $0\leq q_N \leq c$ for a suitable $c$ independent of $N$. 
From $r_0^N=1-r_0$ we obtain with Taylor expansion
\[
  \frac{1}{N} \ln (1-r_0)=\ln r_0 =\ln \big(1+(r_0-1)\big)=r_0-1 +R_N, \quad \text{with} ~|R_N|\leq 2(1-r_0)^2. 
\]
This can be rewritten as $(1-r_0)e^{N(1-r_0)}=e^{NR_N}$. Multiplying by $N$ and using the definition of the function $W$ we thus obtain $N\big( 1- r_0(N)\big) = W(N e^{N R_N})$. Hence
\[
   q_N=\frac{W(N e^{N R_N})}{\ln N}
\]
holds. Now note
\[
 \beta_N:=\frac{\ln (N e^{N R_N})}{\ln N} =1+\frac{N R_N}{\ln N}= 1+q_N \frac{R_N}{1-r_0(N)}.
\]
Using $|q_N| \leq c$, $|R_N|\leq 2(1-r_0(N))^2$ and $\lim_{N \to \infty}  r_0(N)=1$ it follows that $\lim_{N\to \infty} \beta_N =1$ holds. With the left inequality in \eqref{EstiW} we get  
\[ \begin{split}
 q_N &= \frac{W(N e^{N R_N})}{\ln N}\geq \frac{\ln (N e^{N R_N})}{\ln N} -\frac{\ln \ln (N e^{N R_N})}{\ln N}= \beta_N - \frac{ \ln (\beta_N \ln N)}{\ln N}  \\ & = \beta_N - \frac{\ln \ln N}{\ln N} - \frac{\ln \beta_N}{\ln N} \to 1 \quad \text{for}~ N \to \infty.
\end{split} \]
Combining this with the result in \eqref{P1} completes the proof of \eqref{estW1}. 
\end{proof}
\\[1ex]
For $N \geq 4$ it is not possible to derive an explicit expression in terms of elementary functions for $r_0(N)$. We do have a lower estimate $r^\ast(N)$ of $r_0(N)$ as in \eqref{estW2}.
For this estimate we  have, for $N \geq 3$,
\[
  r^\ast(N) \geq 1- \frac{1}{N} \big( \ln N -\tfrac12  \ln \ln N\big) \geq 1- \frac{\ln N}{N}
\]
and thus, it has  the correct asymptotic behavior: $ \lim_{N \to \infty} \frac{r^\ast(N)}{r_0(N)}=1$.  
\begin{corollary}\label{CorolA}
Let $c>0$ be given. For $r^\ast(N)=1-\frac{W(cN)}{N}$ the following holds for all $N \in \Bbb{N}$:
\begin{equation} \label{eqcorol}
  c r^N \leq 1-r  \quad \text{for all}~ r\in [0,r^\ast(N)].
\end{equation}
\end{corollary}
\begin{proof}
This result follows from an easy modification in the first part of the proof of Lemma~\ref{lemmaW1}. Replace $g_N(r)=e^{-N(1-r)}$ by $g_N(r)=ce^{-N(1-r)}$ and all arguments apply, yielding the result \eqref{eqcorol}.
\end{proof}

\section{Bounds for $1-\rmax$} \label{secbounds}
Our aim is  to derive  estimates for $1-\rmax$, cf. \eqref{claim}.  We are mainly interested in (sharp) \emph{upper} bounds for $1- \rmax$, since $1- \rmax \leq c(N)$ implies the positivity property 
\[ \min_{t \in [-1,1]} G_{N}(r,t)  \geq 0 \quad \text{for all}~ r \in [0,1-c(N)].
\]
Below we first derive such upper bounds and we then study the sharpness of these bounds by deriving lower bounds for $1-\rmax$ in Section~\ref{seclower}.
\subsection{Upper bound for $1-\rmax$} \label{secupper}
Using $\|L_{n,d}\|_{L^\infty[-1,1]}=1$ and  \eqref{solution2} we obtain the basic estimate
\begin{equation} \label{centralideaA} \begin{split}
\min_{t \in [-1,1]} G_{N}(r,t) & \geq G(r,-1) - \|G_{N}(r,\cdot)- G(r,\cdot)\|_{L^\infty[-1,1]}   \\ & \geq  G(r,-1)- \sum_{n=N+1}^\infty \nu_{n,d} r^n \\
   & = \frac{1-r_0}{(1+r_0)^{d-1}} - \sum_{n=N+1}^\infty \nu_{n,d} r^n\quad \text{for all}~r \in [0,1).
\end{split} \end{equation}
Using  \eqref{ResClaim} this yields the following criterion for deriving an upper bound for  $1-\rmax$:
\begin{equation} \label{suffcondA}
 (1+r_0)^{d-1}\sum_{n=N+1}^\infty \nu_{n,d} r_0^n \leq 1-r_0 \quad \Rightarrow ~1-\rmax \leq 1- r_0,
\end{equation}
which is used in the proof of the following theorem.  Note that \eqref{suffcondA} can only hold for $r_0 < 1$.
\begin{theorem} \label{mainThm1}
 Define 
 \begin{equation} \label{defBN} B_{N,d}:=\ln\left(\sqrt{C_d (d-2)!}\,\,2^{\frac12(d-3)} N^{\frac12 d} \right), \quad N \geq 1,~ d \geq 2,
 \end{equation}
 with $C_d$ as in \eqref{boundnu}, i.e., $C_2=C_3=2$, $C_d=d$ for $d \geq 4$.  Assume $N \geq 3$. The following holds:
 \begin{equation} \label{mainestimate}
  1- \rmax \leq \min \left \{\frac{2}{N}\left( B_{N,d} +\tfrac14(d-2) B_{N,d}^{-1}\right), \, 1 \right\}.
 \end{equation}
\end{theorem}
\begin{proof}
First note that from \eqref{boundnu} it follows that
\begin{equation} \label{L1}
\sum_{n=N+1}^\infty \nu_{n,d} r^n \leq C_d \sum_{n=N+1}^\infty (n+1)^{d-2} r^n \leq C_d r^{N-1} \sum_{k=2}^\infty (N+k)^{d-2} r^k.
\end{equation}
 We use the binomium 
$
  (N+k)^{d-2}= \sum_{\ell=0}^{d-2} \begin{pmatrix} d-2 \\ \ell \end{pmatrix} N^{d-2 -\ell} k^\ell $
and  the estimate
\begin{equation} \label{esthulp}
  \sum_{k=1}^\infty k^\ell r^k \leq \ell ! \frac{r}{(1-r)^{\ell +1}}, \quad r\in [0,1), ~\ell \in  \Bbb{N}.
\end{equation}
For completeness we give a proof of this estimate in the Appendix. Using \eqref{L1}, \eqref{esthulp}, $\ell ! \leq (d-2)!$ for $\ell \leq d-2$, we  obtain  for arbitrary $r\in [0,1)$:
\begin{align*}
 (1+r)^{d-1}\sum_{n=N+1}^\infty \nu_{n,d} r^n & \leq 2^{d-1}C_d r^{N-1} \sum_{\ell=0}^{d-2} \begin{pmatrix} d-2 \\ \ell \end{pmatrix} N^{d-2 -\ell} \sum_{k=2}^\infty k^\ell r^k  \\ & \leq 2^{d-1}C_d (d-2)! \frac{r^N}{1-r}\sum_{\ell=0}^{d-2} \begin{pmatrix} d-2 \\ \ell \end{pmatrix} N^{d-2 -\ell} \left(\tfrac{1}{1-r}\right)^\ell \\
 & = 2^{d-1}C_d (d-2)! \frac{r^N}{1-r} \left(N+ \tfrac{1}{1-r}\right)^{d-2}.
\end{align*}
It follows that the left inequality in \eqref{suffcondA} is satisfied if $r_0 \in [0,1)$ satisfies  the following inequality, with  the notation $\alpha_d:=2^{d-1}C_d (d-2)!$:
\begin{equation} \label{esthulp1}  \begin{split}
   \alpha_d r_0^N\left(1+ \tfrac{1}{N(1-r_0)}\right)^{d-2}N^{d-2} & \leq (1-r_0)^2  \\
  \Leftrightarrow~~ \alpha_d^\frac12 r_0^{\frac12 N}\left(1+ \tfrac{1}{N(1-r_0)}\right)^{\frac12 d-1} N^{\frac12 d} & \leq N(1-r_0).
  \end{split}
\end{equation}
With $\ln(1+x) \leq x$ for $x >-1$ we obtain
\[ \begin{split}
 r_0^{\frac12 N}\left(1+ \tfrac{1}{N(1-r_0)}\right)^{\frac12 d-1} & = e^{\frac12 N\ln(1+(r_0-1))} e^{\frac12(d-2) \ln (1+\tfrac{1}{N(1-r_0)})} \\ & \leq e^{-\frac12 N(1-r_0)} e^{\frac{\frac12 (d-2)}{N(1-r_0)}}= e^{-y_{N,r_0}}e^{\frac{\frac14(d-2)}{y_{N,r_0}}},
\end{split} \]
 with $y_{N,r_0}:=\tfrac12 N(1-r_0) \in [0,\tfrac12 N)$. This results in the sufficient condition: $r_0=r_0(N) \in [0,1)$ such that
\begin{equation} \label{esthulp2}
  \tfrac12 \alpha_d^\frac12 N^{\frac12 d} e^{\frac{\frac14(d-2)}{y_{N,r_0}}}  \leq y_{N,r_0} e^{y_{N,r_0}}.
\end{equation}
 Due to monotonicity properties of the functions  $z \to e^{\frac{\frac14(d-2)}{z}}$ and $z \to z e^z$ for $  z \geq 0$, it follows that there exists a unique $z_\ast >0$ such that $\tfrac12 \alpha_d^\frac12 N^{\frac12 d} e^{\frac{\frac14(d-2)}{z_\ast}}  =z_\ast e^{z_\ast}$ holds. Furthermore, \eqref{esthulp2} is satisfied for all  $y_{N,r_0} \in [0,\tfrac12 N)$ that satisfy $y_{N,r_0} \geq z_\ast$. Using the Lambert W function it follows that $z_\ast$ is the unique positive solution of
\[
 W\left( \tfrac12 \alpha_d^\frac12 N^{\frac12 d}e^{\frac{\frac14 (d-2)}{z_\ast}} \right) = z_\ast.
\]
Using $\alpha_d \geq 4$ and $N\geq 3$ it follows that $\tfrac12 \alpha_d^\frac12  N^{\frac12 d} e^{\frac{\frac14(d-2)}{z_\ast}} \geq 3$.  We apply  \eqref{EstiW}, which results in
\begin{equation} \label{esthulp3} \begin{split}
 z_\ast & = W\left(\tfrac12 \alpha_d^\frac12 N^{\frac12 d}e^{\frac{\frac14(d-2)}{z_\ast}} \right) \leq \ln \left(\tfrac12 \alpha_d^\frac12 N^{\frac12 d} e^{\frac{\frac14(d-2)}{z_\ast}} \right) \\ &  = \ln\left(\tfrac12 \alpha_d^\frac12  N^{\frac12 d}\right) + \frac{\frac14(d-2)}{z_\ast}=  B_{N,d}  + \frac{\frac14(d-2)}{z_\ast},  
 \end{split}
\end{equation}
cf. \eqref{defBN}. Hence, $z_\ast$ satisfies $z_\ast^2 \leq B_{N,d} \, z_\ast + \tfrac14 (d-2) $.  A monotonicity argument shows $z_\ast \leq \hat z$ where $\hat z$ is the positive root of  the quadratic equation $\hat z ^2 = B_{N,d} \, \hat z + \tfrac14 (d-2)$. With $\sqrt{1+x} \leq 1 + \tfrac12 x$ for $x>0$, we obtain
\begin{equation} \label{P3} \begin{split}
  \hat z & = \tfrac12 \left( B_{N,d} +\sqrt{   B_{N,d}^2 + d-2}\right)   =
 \tfrac12   B_{N,d}\left( 1+  \sqrt{ 1+ (d-2) B_{N,d}^{-2} } \right)
   \\ & \leq   B_{N,d}  + \tfrac14 (d-2)   B_{N,d}^{-1}. 
\end{split} \end{equation}
It follows that  \eqref{esthulp2} is satisfied for $ y_{N,r_0}:=B_{N,d}  + \tfrac14 (d-2)   B_{N,d}^{-1} \geq \hat z \geq z_\ast$, provided $y_{N,r_0} \leq \tfrac12 N$.
Recall that  $r_0$ and $y_{N,r_0}$ are related through $y_{N,r_0}=\tfrac12 N(1-r_0)$.  We conclude that the left inequality in \eqref{suffcondA} is satisfied for $r_0=r_0(N) \in [0,1)$ defined by 
\begin{equation} \label{Est6} 1-r_0 = \min\left\{\frac{2}{N} \left( B_{N,d}  + \tfrac14 (d-2)   B_{N,d}^{-1}\right),\, 1\right\} ,\end{equation}   
which completes the proof.
\end{proof}

\begin{remark} \label{remspecialcase} \rm
We discuss the bound in \eqref{mainestimate}. First we consider the estimate for the important cases $d=2$, $d=3$: 
\begin{align} 
  1- \rmax & \leq  2 \frac{\ln N}{N} \quad \text{for $d=2$},~ N \geq 3, \label{spcase1}\\
  1- \rmax & \leq \frac{2}{N} \left( \ln\big( \sqrt{2}N^\frac32\big)  +\frac{1}{4\ln\big( \sqrt{2}N^\frac32\big)} \right)  \label{spcase2} \quad \text{for $d=3$},\, N \geq 8,  \\
  &  \approx  3 \frac{ \ln\big( 2^\frac13 N \big) }{N}.  \label{spcase3}
\end{align}
In \eqref{spcase3} we replaced $1+\frac14 \ln^{-2}\big( \sqrt{2}N^\frac32\big)$ by 1. Note $\frac14 \ln^{-2}\big( \sqrt{2}N^\frac32\big) \leq 0.021$ for $N\geq 8$.  In \eqref{spcase1} and \eqref{spcase2} we added $N\geq3$ and $N\geq 8$ to guarantee that the upper bounds are smaller than one. \\
We now consider $d \geq 4$, i.e., $C_d=d$ in \eqref{defBN}. As above we use the approximation $1+\tfrac14(d-2)B_{N,d}^{-2} \approx 1$. For the upper bound in \eqref{mainestimate} we then (approximately) have
\[
  \frac{2}{N}B_{N,d} = \frac{2}{N} \ln \left(\sqrt{d (d-2)! }\,2^{\frac12(d-3)} N^{\frac12 d}\right)= \frac{d}{N} \ln (\tilde c_d N), \quad \tilde c_d:=\big(d (d-2)!\big)^{\frac{1}{d}} 2^{1-\frac{3}{d}}.
\]
To see the growth of $\tilde c_d$ as a function of $d$ we use Stirling's approximation to conclude that for large enough $d$ we have $\tilde c_d \approx\big(d!\big)^{\frac{1}{d}} 2 \approx \tfrac{2}{e} d$. Hence for large $d$ we have
\[
 1- \rmax \lesssim \frac{d}{N} \ln \big( \tfrac{2}{e} d N \big).
\]
\end{remark}
 \begin{remark} \label{rem2dspecial} \rm
 A more refined analysis for the case $d=2$, used in \cite{ReuskenDD},  results in a (slightly) better bound as in \eqref{spcase1}. In this analysis we do not use the triangle inequality in \eqref{centralideaA}, but  directly study positiveness of $G_{N}(r,t)$. We outline this analysis. The Legendre polynomials $L_{n,2}$ are the Chebychev polynomials, i.e. $L_{n,2}(t)=T_n(t)=\cos (n \arccos t)$, cf. \eqref{Legendredistwo}. With $z:=re^{i  \theta}$, $\theta = \arccos t$,  one obtains
\begin{equation} \label{sumexpression}
  G_{N}(r,t)=1+\sum_{n=1}^N r^n \cos (n \theta)= \frac{1-r^2 - \big(z^{N+1}  +\bar{z}^{N+1}\big)+ r^2\big(z^{N}  +\bar{z}^{N}\big)}{(1-z)(1-\bar{z})}.
\end{equation}
Since the denominator $(1-z)(1-\bar{z})=1-2 r \cos(\theta)+r^2 \geq (1-r)^2$ is strictly positive, the sign of $G_{N}(r,t)$ is determined by the nominator. Note
\begin{align*}
 & 1-r^2 - \big(z^{N+1}  +\bar{z}^{N+1}\big)+ r^2\big(z^{N}  +\bar{z}^{N}\big) \\
 & = 1-r^2 -2 r^{N+1} \cos \big((N+1) \theta \big) + 2 r^{N+2} \cos(N \theta) \geq 1-r^2 -4 r^{N}.
\end{align*}
Hence,  a sufficient condition for $ \min_{t \in [-1,1] } G_{N}(r,t)\geq 0$ to hold is $r^N \leq \tfrac14(1-r^2)$. With $\hat r:=r^2$ this condition is $4\hat r^{\frac12 N} \leq 1-\hat r$. Using Corollary~\ref{CorolA} it follows that the latter condition is satisfied for all $\hat r \in [0, r^\ast(\tfrac12 N)]$ with $r^\ast(\tfrac12 N)=1-\frac{W(4\cdot \frac12 N)}{\frac12 N}=1- 2 \frac{W(2N)}{N}$. Using this and \eqref{EstiW} implies, for $N \geq 2$:
\begin{equation} \label{spcase1improved}
 1- (\rmax)^2 \leq 2 \frac{\ln (2 N)}{N}.
\end{equation}
For comparing this with \eqref{spcase1} we introduce $b(N):=\frac{\ln N}{N}$. In \eqref{spcase1}  we have the bound $2 b(N)$. From \eqref{spcase1improved} we obtain 
\begin{equation} \label{spcase1improvedA}
 1- \rmax\leq  \frac{ 4 b(2N)}{1+ \sqrt{1-4 b(2N)}}= 2 b(2N)\big( 1+ \mathcal{O}(b(2N))\big)  \quad (N \to \infty ).
\end{equation}
Hence, for not too small $N$ values this  bound is essentially of the form $2 b(2N)$, i.e.,
\begin{equation} \label{boundbetter}
 1-\rmax \lesssim \frac{\ln (2N)}{N},
\end{equation}
a bound in \eqref{spcase1} but with the factor 2 shifted into the logarithmic term. For $d=2$ the Legendre ($=$ Chebychev) polynomial has the simple trigonometric polynomial representation $L_{2,d}(t)=\tfrac12 \big( (e^{i \theta})^n +(e^{-i \theta})^n\big)$ which leads to the simple sum relation in \eqref{sumexpression}. For $d \geq 3$ the Legendre polynomials do not have such a representation. Therefore,   a sign  analysis directly based on $G_{N}(r,t)$, as in \eqref{sumexpression}, has no straightforward generalization for $d \geq 3$.  
 \end{remark}

 \subsection{Lower bound for $1-\rmax$} \label{seclower}
 We study sharpness of the  bound derived in  \eqref{mainestimate}. First note that a lower bound for $1- \rmax$ is given in \eqref{apriori}. This result shows that (for odd $N$) the $\frac{1}{N}$ term in \eqref{mainestimate} can not be avoided. We claim that also the $\ln N$ term in \eqref{mainestimate} can not be avoided. To support this claim we restrict to the (from an application point of view) most important cases $d=2$ and $d=3$ and derive lower bounds for these two cases. In both cases we use $\min_{t \in [-1,1]}  G_N(r,t)\leq  G_{N}(r,t_0)$ for all $t_0 \in [-1,1]$ and thus we have the following criterion for deriving a lower bound for $1-\rmax$:
 \begin{equation} \label{criterion2}
   G_{N}(r_0,t_0) < 0 \quad \text{for $r_0 >0$, $t_0 \in [-1,1]$} ~\Rightarrow ~~1-r_0 < 1-\rmax.
\end{equation}
Below we refine the analysis used in the proof of \eqref{Res5} in Lemma~\ref{lemClaim}. 
We first consider $d=2$.
\begin{lemma} \label{lem2d}
Consider $d=2$. Take $\beta \in (0,1)$ and define $r_0=r_0(\beta,N):= 1- \beta \frac{\ln N}{N}$. There exists $N_0 \in \Bbb{N}$ such that 
\[ \min_{t \in [-1,1]}  G_N(r_0,t) < 0\quad \text{for all}~~ N > N_0.
\]
\end{lemma}
\begin{proof} 
We use $\theta= \arccos t \in [0,\pi]$ and with $z:=r e^{i \theta}$ one obtains the representation
\begin{equation} \label{repres} \begin{split} G_N(r,t) & = \sum_{n=0}^N \nu_{n,2} r^n (-1)^n = -1 +2 \sum_{n=0}^N \cos(n \theta) r^n  \\
& = \frac{1-r^2 - \big(z^{N+1} +\bar z^{N+1}\big)+=\frac{\ln (2N)}{N}r^2\big(z^N +\bar z^N\big)}{(1-z)(1-\bar z)} \\
& = \frac{1-r^2 - 2r^{N+1} \cos\big((N+1)\theta\big) +2 r^{N+2}\cos (N \theta)}{1- 2r \cos(\theta) + r^2}.
\end{split} \end{equation}
Since the denominator is always positive,  the sign of $G_N(r,t)$ is determined by the sign of the nomimator.
First consider odd $N$. We then take $\theta=\pi$, i.e., $t_0=-1$ and thus the nominator equals $(1+r)\big(1-r - 2r^{N+1}\big)$. 
Note that $\lim_{N \to \infty} r_0= 1$ and $\lim_{N \to \infty} N^\beta r_0^N= 1$. Thus we have, using $\beta <1$,
\[ \begin{split}
 & \lim_{N\to \infty} N^\beta\big(1+r_0)(1-r_0-2 r_0^{N+1}\big)= 2 \beta \lim_{N \to \infty} N^{\beta-1} \ln N -4 \lim_{N\to \infty} N^\beta r_0^N =-4.
\end{split} \]
Hence for  $N_0$ sufficienly large $G_N(r_0,-1) < 0$ holds for all odd $N > N_0$. We now consider even $N$ and take $\theta=(1-\tfrac{1}{N})\pi$, i.e., $t_0=\cos(\pi - \frac{\pi}{N})$. For this $\theta$ we have $\cos\big((N+1)\theta\big)= \cos(\tfrac{\pi}{N})$, $\cos (N \theta)=-1$. The nominator then equals $(1+r)\big(1-r - 2r^{N+1}\big) + 2 r^{N+1}\big(1-\cos(\tfrac{\pi}{N})\big)$. Taking $r=r_0$ and noting $\lim_{N\to \infty} \big(1-\cos(\tfrac{\pi}{N})\big)=0$ we obtain
\[
  \lim_{N\to \infty} N^\beta \left((1+r_0)\big(1-r_0-2 r_0^{N+1}\big) +2 r_0^{N+1}\big(1-\cos(\tfrac{\pi}{N})\right)=-4, 
\]
and thus for  $N_0$ sufficienly large $G_N(r_0,t_0) < 0$ holds for all even $N > N_0$.
\end{proof}
\ \\[1ex]
For $d=3$ we use a similar approach, where we now restrict to $N$ odd, cf. Remark~\ref{RemOdd}.
\begin{lemma} \label{lem3d}
Consider $d=3$. Take $\beta \in (0,2)$ and define $r_0=r_0(\beta,N):= 1- \beta \frac{\ln N}{N}$. There exists $N_0 \in \Bbb{N}$ such that 
\[ \min_{t \in [-1,1]}  G_N(r_0,t) < 0\quad \text{for all odd}~~ N > N_0.
\]
\end{lemma}
\begin{proof}
Note that $\sum_{n=1}^N n x^n= x \frac{d}{dx}\left(\frac{1- x^{N+1}}{1-x}\right)= \frac{x-(N+1)x^{N+1}+N x^{N+2}}{(1-x)^2}$ ($x \neq 1$). Using this we obtain for odd $N$:
\begin{align*} G_N(r,-1)& = \sum_{n=0}^N \nu_{n,3} r^n (-1)^n =  \sum_{n=0}^N (2n+1)(-r)^n = 2 \sum_{n=1}^N n (-r)^n + \sum_{n=0}^N (-r)^n \\
 & = \frac{1-r - 2(N+\tfrac32) r^{N+1} - 2 (N + \tfrac12) r^{N+2}}{(1+r)^2}.
\end{align*}
For $r=r_0$ the sign is determined by the sign of the nomimator. 
Note that $\lim_{N \to \infty} r_0= 1$ and $\lim_{N \to \infty} N^\beta r_0^N= 1$. Thus we have, using $\beta <2$,
\begin{align*} & \lim_{N\to \infty} N^{\beta-1}\big(1-r_0-2(N+\tfrac32) r_0^{N+1} - 2 (N + \tfrac12) r_0^{N+2}\big)  =  -4.
\end{align*}
Hence for  $N_0$ sufficienly large $G_N(r_0,-1) < 0$ holds for all odd $N > N_0$.
\end{proof}

\begin{remark} \label{RemOdd}
\rm In the two lemmas above the result for \emph{odd} $N$ is obtained by taking $t_0=-1$ in \eqref{criterion2} and using that $L_{n,d}(-1)=(-1)^n$, which holds for all dimensions $d \geq 2$. For even $N$, we use in the proof of Lemma~\ref{lem2d} the explicit sum formula \eqref{repres} which is based on the Chebychev polynomial property $L_{n,2}(\cos \theta)=\cos( n \theta)$. This explicit formula then allows the choice of a suitable $t_0$ ($\neq -1$) such that asymptotically $G_N(r_0,t_0) < 0$ holds. This approach can not be used for $d \geq 3$. Therefore, in Lemma~\ref{lem3d} we restrict to odd $N$. A related interesting observation is that in  numerical experiments, cf. Figure~\ref{FigRes3D} below, we observe  that  for $d=3$ there is a significant difference in behavior of $\rmax$ between even and odd $N$, whereas for $d=2$ this is not the case. 
\end{remark}

\begin{corollary} \label{corol1} \rm 
We compare upper and lower bounds. For $d=2$ we have $1- \rmax \leq 2 \frac{\ln N}{N}$, which holds
for \emph{all} $N \geq 3$. This bound can be improved to a bound that essentially behaves as $\frac{\ln (2N)}{N}$, cf. \eqref{spcase1improvedA}. We have an \emph{asymptotic} lower bound of the form $\beta\frac{\ln N}{N} \leq 1- \rmax$ for any $\beta \in (0,1)$. This shows that since $\lim_{N \to \infty} \frac{\ln (2 N)}{N}/\frac{\ln N}{ N} =1$ the improved upper bound has the correct asymptotic behavior. The first upper bound is asymptotically a factor 2 too pessimistic, which is related to the fact that we use a triangle inequality and the bound holds for all $N \geq 3$. For $d=3$ we have the upper bound $3 \frac{ \ln (2^\frac13 N)}{N}$, cf. \eqref{spcase3}, that  holds for all $N \geq 8$. We have an \emph{asymptotic} lower bound of the form $\beta\frac{\ln N}{N} \leq 1- \rmax$ for any $\beta \in (0,2)$. Hence, since $\lim_{N \to \infty} \frac{3\ln (2^\frac13 N)}{N}/\frac{2\ln N}{ N}=\tfrac32$, the asymptotic behavior of the upper bound is correct apart from a factor $\frac32$. 
\end{corollary}

\section{Numerical experiments} \label{SecExp}
We numerically study for $d=2$ and $d=3$ the bounds derived for $1- \rmax$. For this we numerically determine, for given $N$, $\min \{\,\delta\,|\, \min_{|t|\leq 1} G_N(r,t) \geq 0 \quad \text{for all}~r \in [0,1-\delta] \,\}$. This approximation is denoted by  $\delta_{\rm num}(N)$ and yields an accurate estimate for $1-\rmax$, cf. Remark~\ref{Explain}.
\begin{remark} \label{Explain} \rm 
We explain, how we determine $\delta_{\rm num}(N)$ for a given $N \in [5,200]$.   Based on \eqref{Res5}, \eqref{spcase1}, \eqref{spcase3}, we define an a-priori interval $I_N=[a_N,b_N]$, $a_N:=1-2 \frac{\ln N}{N}$ for $d=2$, $a_N=1-3 \frac{\ln (2^{\frac13} N)}{N}$ for $d=3$, $b_N:=1-\frac{1}{2N}$ for which $\rmax \in I_N$ holds. We discretize this $r$-interval with $r_i:=a_N+(i-1)\frac{b_N-a_N}{100}$, $i=1,\ldots 101$. The $t$-interval is discretized using $M=20N$ equidistant points $t_j:=-1+(j-1)\frac{2}{M-1}$, $1 \leq j \leq M$. For given $r_i$ we determine $m_i:=\min_{1 \leq j \leq M} G_N(r_i,t_j)$, where the discrete Poisson kernel $G_N$  is defined in \eqref{idGN}, with $L_{n,2}(t)=T_n(t)= \cos (n \arccos t)$ and $L_{n,3}(t)=L_n(t)$ the classical Legendre polynomials that are evaluated using their recursive definition.  Starting from $m_1 >0$ we take $i=2,3,\ldots, i_{\rm crit}$, with $i_{\rm crit}$ the smallest $i$ value for which $m_i <0$ holds.  Finally, we define $\delta_{\rm num}(N):= 1- r_{i_{\rm crit}-1}$. 
\end{remark}

 In the experiments we use values $N=N_k=5k,~k=1,\ldots,40$. We compare $\delta_{\rm num}$ with the upper bounds, in particular:
\begin{align}
d=2: & \quad 1-\rmax \leq 2\frac{\ln N}{N}=:\delta_{{\rm th},1}(N) \quad \text{(cf. \eqref{spcase1})},\\
d=2: & \quad 1-\rmax \leq \frac{\ln (2 N)}{N}=:\delta_{{\rm th},2}(N),\quad \text{( cf. \eqref{boundbetter})},\\
d=3: & \quad 1-\rmax \leq 3\frac{ \ln (2^\frac13 N)}{N}=:\delta_{{\rm th}}(N), ~~N \geq 8,\quad\text{(cf. \eqref{spcase3})}.
\end{align}
For a better comparison  it is convenient to consider the inverses, i.e., $1/\delta_{\rm th}$ and $1/\delta_{\rm num}$. We also use an ansatz $\delta(a,b,N):= b \frac{\ln (a N)}{N}$ and fit this, using a weighted least squares method, to the numerical estimate $\delta_{\rm num}(N)$ of $1- \rmax$ as follows. The data are denoted by ${\bf b}=\left( \delta_{\rm num}(N_1), \ldots, \delta_{\rm num}(N_{40})\right) \in \R^{40}$. We have
\[
  \delta(a,b,N)= b \frac{\ln (a)}{N}+ b \frac{\ln N}{N}= \tilde a \frac{1}{N} + b   \frac{\ln N}{N}, \quad \tilde a=b \ln a.
\]
 We consider larger $N$ values to be more relevant than the very small ones and want to avoid that an initialization effect pollutes the parameter choice. Therefore we use a weighted least squares with a weighting $N_i$ for the $i$th measurement. Hence, we solve the linear least squares problem
 \[
     \min_{\tilde a, b} \sum_{i=1}^{40} \left( \tilde a + b \ln (N_i) - N_i {\bf b}_i\right)^2.
 \] 
The resulting parameter values are denoted by $\tilde a_\ast$, $b_\ast$, from which we obtain $a_\ast:=e^{\tilde a_\ast/b_\ast}$. This finally yields the fitted curve $\delta(a_\ast,b_\ast, x)=b_\ast \frac{\ln (a_\ast x)}{x}$, $x >0$. The resulting fitted curves are:
\begin{equation} \label{fit} \begin{split}
  d=2: & \quad \delta(a_\ast,b_\ast,N)= 0.84 \frac{\ln (1.03 N)}{N} \approx 1-\rmax,\\
  d=3: & \quad \delta(a_\ast^{\rm o},b_\ast^{\rm o},N)= 1.99 \frac{\ln (0.59 N)}{N}\approx 1-\rmax, \quad N \quad \text{odd}, \\
  d=3: & \quad \delta(a_\ast^{\rm e},b_\ast^{\rm e},N)= 1.73 \frac{\ln (0.63 N)}{N}\approx 1-\rmax, \quad N \quad \text{even}.
\end{split}
\end{equation}
Results are presented in the Figures~\ref{FigRes2D},\ref{FigRes3D}.

\begin{figure}[ht!]
	\centering\includegraphics[width=0.9\textwidth]{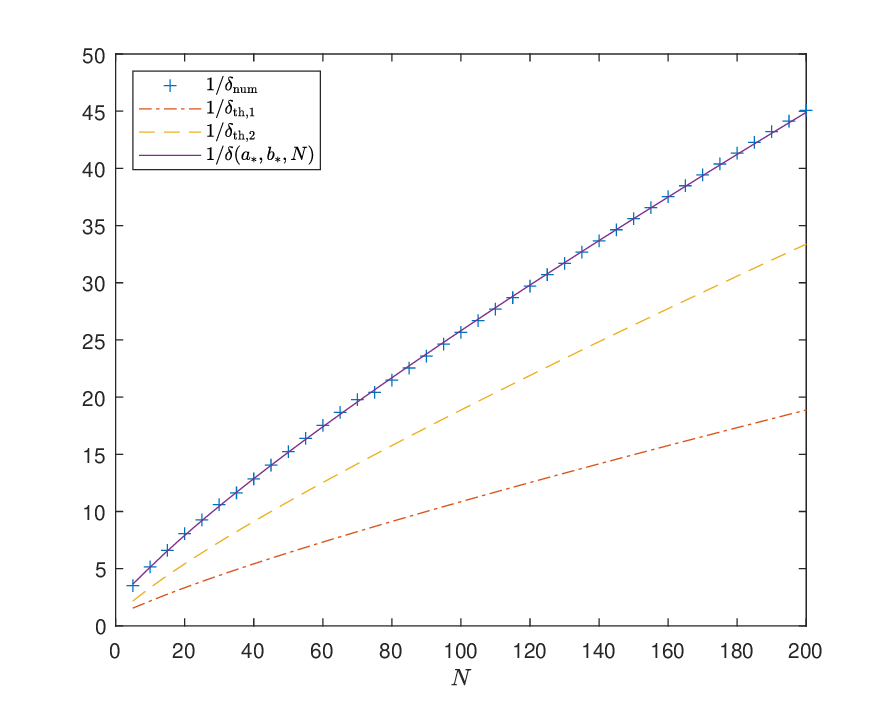} \hspace*{-0.8cm}
	\caption{Results for $d=2$: $1/\delta_{\rm num}$ and least squares fit $1/\delta(a_\ast,b_\ast, N)$, theoretical bounds $1/\delta_{{\rm th},1}$, $1/\delta_{{\rm th},2}$. \label{FigRes2D}}
\end{figure}

\begin{figure}[ht!]
	\centering\includegraphics[width=0.9\textwidth]{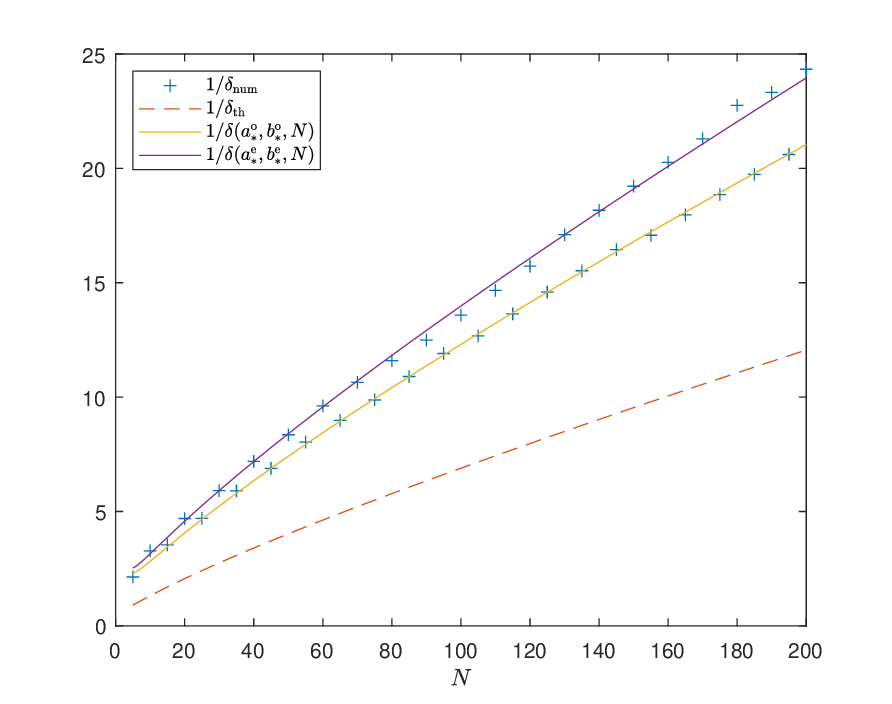} \hspace*{-0.8cm}
	\caption{Results for $d=3$: $1/\delta_{\rm num}$ and least squares fit $\delta(a_\ast^{\rm o},b_\ast^{\rm o}, N)$ for $N $ odd, $\delta(a_\ast^{\rm e},b_\ast^{\rm e}, N)$ for $N $ even, and theoretical bound $1/\delta_{{\rm th}}$. \label{FigRes3D}}
\end{figure}
These results confirm the theoretical analysis. Note that the curves defined in \eqref{fit}  yield a very good fit to (the approximation of) $1- \rmax$ and that for $d=3$ there is a difference between $N$ even and $N$ odd. 
\section{Appendix}
We give a proof of the estimate \eqref{esthulp}.  We use the binomial series
\[
  \frac{1}{(1-r)^{\ell +1}}= \sum_{k=0}^{\infty} \begin{pmatrix} \ell + k \\ k \end{pmatrix} r^k, \quad |r|<1,~\ell \in \Bbb{N}.
\]
With $\begin{pmatrix} \ell + k \\ k \end{pmatrix}=\begin{pmatrix} \ell + k \\ \ell \end{pmatrix} \geq \frac{(k+1)^\ell}{\ell !}$ we obtain
\[
 \frac{r}{(1-r)^{\ell +1}} \geq \frac{1}{\ell !} \sum_{k=0}^{\infty} (k+1)^\ell r^{k+1}= \frac{1}{\ell !} \sum_{k=1}^{\infty} k^\ell r^{k},
\]
which is the inequality \eqref{esthulp}.
\ \\[1ex]
{\bf Acknowledgements}\\
The author  acknowledges funding by the Deutsche
Forschungsgemeinschaft (DFG, German Research Foundation) - Project number 442047500 through
the Collaborative Research Center "Sparsity and Singular Structures" (SFB 1481). Furthermore he thanks Sven Gro{\ss} for helpful discussions.

\bibliographystyle{siam}
\bibliography{literatur}{}

\end{document}